\documentclass[12pt]{amsart}
\usepackage{amsmath,amsthm,amsfonts,amssymb,amscd,euscript}
\usepackage{color}
\usepackage{graphics}
\input{xy}
\xyoption{all}
\newlength{\defbaselineskip}
\theoremstyle{plain}

\theoremstyle{definition}

\theoremstyle{plain}
\newtheorem{thm}{Theorem}

\newtheorem{lem}[thm]{Lemma}

\newtheorem{prop}[thm]{Proposition}

\theoremstyle{definition}

\newtheorem{assume}[thm]{Assumptions}

\numberwithin{equation}{section}

\begin{document}
\title{Plane automorphisms given by polynomials of scattered degrees}
\author{Kyungyong Lee}
\thanks{Research partially supported by NSF grant DMS 0901367.}

\address{\noindent Department of Mathematics, Wayne State University, Detroit, MI 48202, USA}
\email{klee@math.wayne.edu}

\begin{abstract}
We study the plane automorphisms given by polynomials with certain degree decompositions.
\end{abstract}

 \maketitle
 


\section{Introduction}

The jacobian conjecture, raised by Keller \cite{Keller}, has been studied by many mathematicians: a partial list of related results includes \cite{AM},\cite{BCW},\cite{BK}\cite{NVC},\cite{CW2},\cite{Dru},\cite{EssenWZ},\cite{Gwo},\cite{Hub},\cite{LM},\cite{Mag},\cite{MU}, \cite{Moh},\cite{Nagata},\cite{Wang},\cite{Yag},\cite{Yu}.  A survey is given in \cite{Essen}. In this paper we exclusively deal with the plane case. This is the first in a series of papers that introduce a new computational approach, which is completely elementary.

Let $k$ be a field of characteristic 0, and let $R=k[x,y]$.
Throughout the paper, let $f, g \in R$ be polynomials satisfying the following :  
\begin{assume} Let $n$ be any positive integer, and let $\{d_1,...,d_n\}$ be any set of $n$ distinct positive integers with the following property :
\begin{equation}\label{0422eq1}
\text{ if }d_{i} +d_{j} = d_{p}+d_{q}\text{ for }(\text{not necessarily distinct})\,\, i,j,p,q,  \text{ then }\{d_i,d_j\}=\{d_p,d_q\}.
\end{equation}
For each $d_i\in \{d_1,...,d_n\}$, let $f_{d_i}$ and $g_{d_i}$ be homogeneous polynomials of degree $d_i$ in $R$, and assume that at least one of $f_{d_i}$ and $g_{d_i}$ is nonzero. Let $f=\sum_{i=1}^n f_{d_i}$ and $g=\sum_{i=1}^n g_{d_i}$.  Let $$J=\frac{\partial f}{\partial x} \frac{\partial g}{\partial y}  - \frac{\partial f}{\partial y} \frac{\partial g}{\partial x}=\sum_{1\leq i,j\leq n} \left(\frac{\partial f_{d_i}}{\partial x} \frac{\partial g_{d_j}}{\partial y}  - \frac{\partial f_{d_i}}{\partial y} \frac{\partial g_{d_j}}{\partial x}\right).$$
Denote by $J_{i,j}$ the coefficient of $x^{i}y^{j}$ in $J$.
\end{assume}

\begin{thm}\label{mainthm}
If $J\in k\setminus \{0\}$, then $k[x,y]=k[f,g]$.
\end{thm}

 If $J\in k\setminus \{0\}$, then $f$ and $g$ must have linear parts, i.e., $1\in\{d_1,...,d_n\}$. Assume $d_n=1$. By linear change of variables, we can assume that $f_{d_n}=x$ and $g_{d_n}=y$. The case of $n=2$, where (\ref{0422eq1}) is trivially satisfied, has been proved in \cite{AdEs},  \cite[Corollary 6]{CW1}, \cite{CZ}, \cite{Kire},  \cite{MO} for the plane case under the assumption 
\begin{equation}\label{sqassume}
\left(\begin{array}{cc} \frac{\partial (f-x)}{\partial x} & \frac{\partial (f-x)}{\partial y}\\ & \\ \frac{\partial (g-y)}{\partial x} & \frac{\partial (g-y)}{\partial y}  \end{array}  \right)^2=0,
\end{equation}
and in \cite[Corollary 2.2]{EssenTutaj} without assuming (\ref{sqassume}). The case of  $d_1,...,d_n\leq 100$ is obtained as a special case of a result of Moh \cite{Moh}.

\section{Proof}

For any $d_i \in\{d_1,...,d_{n-1}\}$, let $f_{d_i}= \sum_{j=0}^{d_i} s_{d_i-j,j}x^{d_i-j}y^{j}$
 and $
g_{d_i}=\sum_{j=0}^{d_i} t_{d_i-j,j}x^{d_i-j}y^{j}.
$ Thanks to (\ref{0422eq1}), the coefficients of $x^{d_i-j}y^{j-1}$ in $\left(\frac{\partial f_{d_i}}{\partial x} \frac{\partial g_{d_n}}{\partial y}  - \frac{\partial f_{d_i}}{\partial y} \frac{\partial g_{d_n}}{\partial x}\right)$  and $\left(\frac{\partial f_{d_n}}{\partial x} \frac{\partial g_{d_i}}{\partial y}  - \frac{\partial f_{d_n}}{\partial y} \frac{\partial g_{d_i}}{\partial x}\right)$ are the only ones that contribute to the coefficient $J_{d_i-j, j-1}$, which is equal to 
\begin{equation}\label{0422eq2}
(d_i-j+1)s_{d_i-j+1,j-1} + j t_{d_i-j,j}
\end{equation}for $1\leq j\leq d_i$. 
Since $J\in k\setminus \{0\}$ implies
$(\ref{0422eq2})=0$, there is an element, say $c_{d_i-j+1,j}$, in $k$ such that $s_{d_i-j+1,j-1}=jc_{d_i-j+1,j}$ and $t_{d_i-j,j}=-(d_i-j+1)c_{d_i-j+1,j}.$

By letting $s_{0,d_i}=(d_i+1)c_{0,d_i+1}$ and $t_{d_i,0}=(d_i+1)c_{d_i+1,0}$, we have
$$
f= \sum_{i=1}^{n-1}\sum_{j=1}^{d_i+1} jc_{d_i-j+1,j}x^{d_i-j+1}y^{j-1} + x, \text{ and }
$$
$$
g=-\sum_{i=1}^{n-1}\sum_{j=0}^{d_i} (d_i-j+1)c_{d_i-j+1,j}x^{d_i-j}y^{j}  + y.
$$
In section~\ref{proofsection}, we will show that any $2\times 2$ minor of the matrix
\begin{equation}\label{0424matrix}\left(\begin{array}{cccccccc}
c_{d_i,1} & 2c_{d_i-1,2} & \cdots & (d_i+1)c_{0,d_i+1} & c_{d_j,1} & 2c_{d_j-1,2} & \cdots & (d_j+1)c_{0,d_j+1}\\
 (d_i+1)c_{d_i+1,0} & d_ic_{d_i,1} & \cdots & c_{1,d_i}  &(d_j+1)c_{d_j+1,0} & d_jc_{d_j,1} & \cdots & c_{1,d_j}
\end{array}\right)\end{equation}is equal to $0$ for any $d_i,d_j\in \{d_1,...,d_{n-1}\}$.

Suppose that all $2\times 2$ minors of (\ref{0424matrix}) are 0. Then it is straightforward to prove Theorem~\ref{mainthm} as follows.

\noindent Case 1. Suppose that $c_{d_i+1,0}=0$ for some $d_i$. Then $c_{d_i,1}=\cdots=c_{1,d_i}=0$. Since at least one of $f_{d_i}$ and $g_{d_i}$ is nonzero, $c_{0,d_i+1}\neq 0$. Then  
$c_{d_j+1,0}=c_{d_j,1}=\cdots=c_{1,d_j}=0$ for any $d_j\in \{d_1,...,d_{n-1}\}$. Again since $f_{d_j}\neq 0$ or $g_{d_j}\neq 0$, we get $c_{0,d_j+1}\neq 0$. So we have 
$$
f=\sum_{i=1}^{n-1} (d_i+1)c_{0,d_i+1}y^{d_i} +x \,\,\,\,\,\,\text{ and }\,\,\,\,\,\, g=y.
$$Then $x=f-\sum_{i=1}^{n-1} (d_i+1)c_{0,d_i+1}g^{d_i}$ and $y=g$, hence $k[f,g]=k[x,y]$.
 
\noindent Case 2. Suppose that $(d_i+1)c_{d_i+1,0}\neq 0$ and $c_{d_i,1}=0$ for some $d_i$. Then we use the same argument as in Case 1, and get $$
g=-\sum_{i=1}^{n-1} (d_i+1)c_{d_i+1,0}x^{d_i} +y \,\,\,\,\,\,\text{ and }\,\,\,\,\,\, f=x.
$$

\noindent Case 3. Suppose that $(d_i+1)c_{d_i+1,0}\neq 0$ and $c_{d_i,1}\neq 0$ for any $d_i\in \{d_1,...,d_{n-1}\}$. Then we obtain $$f_{d_i}=c_{d_i,1}\left( x +  \frac{c_{d_i,1}}{(d_i+1)c_{d_i+1,0}}y\right)^{d_i}\,\,\,\,\,\,\text{ and }\,\,\,\,\,\,g_{d_i}=-(d_i+1)c_{d_i+1,0}\left( x +  \frac{c_{d_i,1}}{(d_i+1)c_{d_i+1,0}}y\right)^{d_i}.$$ Note that $\frac{c_{d_i,1}}{(d_i+1)c_{d_i+1,0}}=\frac{c_{d_j,1}}{(d_j+1)c_{d_j+1,0}}$  for any $d_i,d_j\in \{d_1,...,d_{n-1}\}$. So
$$
f=x+\sum_{j=1}^{n-1} c_{d_j,1}\left( x +  \frac{c_{d_j,1}}{(d_j+1)c_{d_j+1,0}}y\right)^{d_j}, \,\,\,\,\,\,\text{ and }
$$
$$
g=y-\sum_{j=1}^{n-1} (d_j+1)c_{d_j+1,0}\left( x +  \frac{c_{d_j,1}}{(d_j+1)c_{d_j+1,0}}y\right)^{d_j}.
$$
Then it is easy to check that 
$$
x=f-\sum_{j=1}^{n-1} c_{d_j,1}\left( f +  \frac{c_{d_j,1}}{(d_j+1)c_{d_j+1,0}}g\right)^{d_j},\,\,\,\,\,\,\text{ and }
$$
$$
y=g+\sum_{j=1}^{n-1} (d_j+1)c_{d_j+1,0}\left( f +  \frac{c_{d_j,1}}{(d_j+1)c_{d_j+1,0}}g\right)^{d_j}.
$$

\section{The vanishing of $2\times 2$ minors of (\ref{0424matrix})}\label{proofsection}
For any $p \in \{1,...,n-1\}$, let\begin{equation*}\label{0424matrix2}
A_p=\left(\begin{array}{cccc}
c_{d_p,1} & 2c_{d_p-1,2} & \cdots & (d_p+1)c_{0,d_p+1} \\
 (d_p+1)c_{d_p+1,0} & d_pc_{d_p,1} & \cdots & c_{1,d_p} 
\end{array}\right).\end{equation*}
Let $(A_p)_{(i,j)}$ be the determinant of the submatrix of $A_p$ obtained by concatenating the $i$-th and $j$-th columns, that is, 
$$(A_p)_{(i,j)}=
\det\left(\begin{array}{cc} ic_{d_p-i+1,i} &jc_{d_p-j+1,j} \\(d_p+2-i)c_{d_p-i+2,i-1} & (d_p+2-j)c_{d_p-j+2,j-1}\end{array}\right).$$
For simplicity, let $d=d_p$, $A=A_p$ and $A_{(i,j)}=(A_p)_{(i,j)}$. 
\begin{prop}\label{prop1}
Any $2\times 2$ minor of $A$ is equal to 0.
\end{prop}
\begin{proof}
This is an immediate consequence of the following two lemmas.
\end{proof}

\begin{lem}\label{0425lem1}
Let $m$ be any positive integer $\leq d+1$. 
If $A_{(1,j)} =0$ for $j\in \{1,...,m\}$, then $A_{(i,j)}=0$ for $i,j\in \{1,...,m\}$.
\end{lem}
\begin{proof}
We use induction on $m$. If $m=1$ then trivial. Suppose that the statement holds for $m-1$, and that $A_{(1,m)}=0$.  If at least one of $c_{d,1}$, $c_{d+1,0}$, $c_{d-m+1,m}$, $c_{d-m+2,m-1}$ is  equal to $0$, then it is easy to show that $c_{d,1}=\cdots=c_{d-m+2,m-1}=0$ and that $c_{d+1,0}=0$ or  $c_{d-m+1,m}=0$. If not, $\left(\begin{array}{c} mc_{d-m+1,m}\\(d+2-m)c_{d-m+2,m-1}\end{array}\right)$ is a multiple of $\left(\begin{array}{c}c_{d,1}\\(d+1)c_{d+1,0}\end{array}\right)$, so $A_{(2,m)}=\cdots=A_{(m-1,m)}=0$ follows from $A_{(2,1)}=\cdots=A_{(m-1,1)}=0$ . 
\end{proof}

\begin{lem}\label{0425lem2}
Let $m$ be any positive integer $\leq d+1$. Then $A_{(1,j)} =0$ for $j\in \{1,...,m\}$.
\end{lem}
\begin{proof}
We use induction on $m$. If $m=1$ then trivial. Suppose that the statement holds for $m-1$. Due to (\ref{0422eq1}),  $J_{2d-m,m-2}$ is the coefficient of $x^{2d-m}y^{m-2}$ in $\frac{\partial f_{d}}{\partial x} \frac{\partial g_{d}}{\partial y}  - \frac{\partial f_{d}}{\partial y} \frac{\partial g_{d}}{\partial x}$.  Looking at 
$$\aligned 
\frac{\partial f_{d}}{\partial x}&= 1\cdot d c_{d,1} x^{d-1} + 2(d-1)c_{d-1,2}x^{d-2}y+3(d-2)c_{d-2,3}x^{d-3}y^2+\cdots,\\
-\frac{\partial g_{d}}{\partial y}&= 1\cdot dc_{d,1}x^{d-1} + 2(d-1)c_{d-1,2}x^{d-2}y+3(d-2)c_{d-2,3}x^{d-3}y^2+\cdots,\\
\frac{\partial f_{d}}{\partial y}&= 1\cdot 2 c_{d-1,2} x^{d-1} + 2\cdot 3 c_{d-2,3}x^{d-2}y+3\cdot 4 c_{d-3,4}x^{d-3}y^2+\cdots,\\
-\frac{\partial g_{d}}{\partial x}&=d(d+1)c_{d+1,0}x^{d-1}  +  (d-1)dc_{d,1}x^{d-2}y+ (d-2)(d-1)c_{d-1,2}x^{d-3}y^2+\cdots,  \endaligned$$
we see that 
$$J_{2d-m,m-2}= -\sum_{i=1}^{m-1} (d-i+1)(m-i)A_{(i,m-i+1)}$$ By induction and Lemma~\ref{0425lem1}, we have $A_{(2,m-1)}=\cdots=A_{(m-1,2)}=0$. Since $J_{2d-m,m-2}=0$, we obtain $A_{(1,m)} =0$.
\end{proof}

Now we will prove that $2\times 2$ minors of (\ref{0424matrix}) are all 0.
Fix two distinct integers $p,q \in \{1,...,n-1\}$. Let $B_{(i,j)}$ be the determinant of the matrix obtained by concatenating the $i$-th column in $A_p$ and $j$-th column in $A_q$, that is, 
$$B_{(i,j)}=
\det\left(\begin{array}{cc} ic_{d_p-i+1,i} &jc_{d_q-j+1,j} \\(d_p+2-i)c_{d_p-i+2,i-1} & (d_q+2-j)c_{d_q-j+2,j-1}\end{array}\right).$$
For simplicity, let $d=d_p$ and $e=d_q$. 

\begin{lem}
$
B_{(d+1,e+1)}=B_{(1,1)}=0.
$
\end{lem}\begin{proof} This is because
$$\aligned
&B_{(d+1,e+1)}^2\\
=&((d+1)c_{0,d+1} c_{1,e} - (e+1)c_{1,d}c_{0,e+1})^2\\
=&(ec_{1,e}^2 -2(e+1)c_{0,e+1}c_{2,e-1})\frac{(d+1)^2}{e}c_{0,d+1}^2\\
&+(dc_{1,d}^2 -2(d+1)c_{0,d+1}c_{2,d-1})\frac{(e+1)^2}{d}c_{0,e+1}^2\\
&+ ( 2(d+1)dc_{0,d+1}c_{2,e-1} - 2dec_{1,d}c_{1,e} + 2(e+1)ec_{2,d-1}c_{0,e+1})\frac{(d+1)(e+1)}{de}c_{0,d+1}c_{0,e+1}\\
=&(A_q)_{(e,e+1)}\frac{(d+1)^2}{e}c_{0,d+1}^2+(A_p)_{(d,d+1)}\frac{(e+1)^2}{d}c_{0,e+1}^2+J_{0,d+e-2}\frac{(d+1)(e+1)}{de}c_{0,d+1}c_{0,e+1}\\
=&0,
\endaligned$$where the last equality follows from Proposition~\ref{prop1} and $J_{0,d+e-2}=0$. Similarly (by symmetry of indices), we obtain $B_{(1,1)}=0$.
\end{proof}

For any $m\in\{1,...,e+1\}$, it is elementary to check that all $2\times 2$ minors of 
$$
\left(\begin{array}{cc}
dB_{(d+1,m)} & B_{(d,m)}\\
(d-1)B_{(d,m)} & 2B_{(d-1,m)} \\
(d-2)B_{(d-1,m)} &   3B_{(d-2,m)} \\
\vdots & \vdots \\
B_{(2,m)} &   dB_{(1,m)} \\
\end{array}\right)$$
are equal to 0.
For example, $$\aligned
&2dB_{(d+1,e+1)}B_{(d-1,e+1)}\\
=&2d((d+1)c_{0,d+1}c_{1,e}-(e+1)c_{1,d}c_{0,e+1}) ((d-1)c_{2,d-1}c_{1,e}-3(e+1)c_{3,d-2}c_{0,e+1})\\
=&d(d-1)(2(d+1)c_{0,d+1}c_{2,d-1} - dc_{1,d}c_{1,d} )c_{1,e}c_{1,e}\\
& + (d-1)(dc_{1,d}c_{1,e}-2(e+1)c_{2,d-1}c_{0,e+1})(dc_{1,d}c_{1,e}-2(e+1)c_{2,d-1}c_{0,e+1})\\
& + 2d(e+1)((d-1)c_{2,d-1}c_{1,d} - 3(d+1)c_{3,d-2}c_{0,d+1} )c_{0,e+1}c_{1,e}\\
& + 2(e+1)^2(3dc_{1,d}c_{3,d-2}-2(d-1)c_{2,d-1}c_{2,d-1})c_{0,e+1}c_{0,e+1}\\
=&-d(d-1)(A_p)_{(d,d+1)}c_{1,e}c_{1,e}\\
&+ (d-1)B_{(d,e+1)}B_{(d,e+1)}\\
& + 2d(e+1)(A_p)_{(d-1,d+1)}c_{0,e+1}c_{1,e}\\
& - 2(e+1)^2(A_p)_{(d-1,d)}c_{0,e+1}c_{0,e+1}
\endaligned$$
implies that  $2dB_{(d+1,e+1)}B_{(d-1,e+1)} - (d-1)B_{(d,e+1)}B_{(d,e+1)}=0$, which is a consequence of Proposition~\ref{prop1}.

Then $
B_{(d+1,e+1)}=0
$ implies that $
B_{(i,e+1)}=0$ for all $2\leq i\leq d$. Note that the coefficients of $y^{d+e-2}$ in  $\left(\frac{\partial f_{d}}{\partial x} \frac{\partial g_{e}}{\partial y}  - \frac{\partial f_{d}}{\partial y} \frac{\partial g_{e}}{\partial x}\right)$  and $\left(\frac{\partial f_{e}}{\partial x} \frac{\partial g_{d}}{\partial y}  - \frac{\partial f_{e}}{\partial y} \frac{\partial g_{d}}{\partial x}\right)$ are the only ones that contribute to $J_{0,d+e-2}$, because of (\ref{0422eq1}). Then we can see that  $J_{0,d+e-2}$ is a linear combination of $B_{(d+1,e)}$ and $B_{(d,e+1)}$ with nonzero coefficients, so we get $B_{(d+1,e)}=0$. Then this implies that $
B_{(i,e)}=0$ for all $2\leq i\leq d$. Since $J_{1,d+e-3}$ is a linear combination of $B_{(d+1,e-1)}$, $B_{(d,e)}$ and $B_{(d-1,e+1)}$ with nonzero coefficients, we get $B_{(d+1,e-1)}=0$. Repeating this argument, we get $
B_{(i,j)}=0$ for all $2\leq i\leq d+1$ and $1\leq j\leq e+1$. Similarly (by symmetry of indices), we obtain $B_{(1,j)}=0$ for all  $1\leq j\leq e+1$. The proof is completed.

\end{document}